\documentclass[12pt,twoside,reqno]{amsart}

\usepackage{amsmath, amssymb, amsthm, enumitem, esint,quiver}
\usepackage[hyperindex]{hyperref}

\usepackage{hyperref}
\hypersetup{bookmarksdepth=3}

\newcommand{\IR}{\mathbb{R}}

\newcommand{\IC}{\mathbb{C}}

\newcommand{\IZ}{\mathbb{Z}}

\newcommand{\IF}{\mathbb{F}}
\newcommand{\NN}{\mathcal{N}}

\newcommand{\RR}{\mathcal{R}}
\renewcommand{\SS}{\mathcal{S}}

\newcommand{\eps}{\varepsilon}
\newcommand{\la}{\lambda}

\newcommand{\td}[1]{\widetilde{#1}}

\DeclareMathOperator{\loc}{loc}

\DeclareMathOperator{\Var}{Var}

\DeclareMathOperator{\Ric}{Ric}
\DeclareMathOperator{\Rm}{Rm}

\newcommand{\rrm}{r_{\Rm}}

\newcommand{\EMPTY}[1]{}

\newtheorem{Theorem}[equation]{Theorem}

\newtheorem{Corollary}[equation]{Corollary}

\newtheorem{Claim}[equation]{Claim}

\theoremstyle{definition}

\theoremstyle{remark}

\newtheorem{Example}[equation]{Example}

\numberwithin{equation}{section}

\title{On the fundamental group of non-collapsed ancient Ricci flows}
\author{Richard H  Bamler}
\address{Department of Mathematics, UC Berkeley, CA 94720, USA}
\email{rbamler@berkeley.edu}
\thanks{This work was supported by NSF grant DMS-1906500.}
\date{\today}

\begin{document}

\begin{abstract}
We show that any manifold admitting a non-collapsed, ancient Ricci flow must have finite fundamental group.
This generalizes what was known for \emph{$\kappa$-solutions} in dimensions 2, 3.
We furthermore show that this fundamental group must be a quotient of the fundamental group of the regular part of any tangent flow at infinity.
\end{abstract}

\maketitle

\section{Introduction}
A central goal in the study of Ricci flow is the study of its singularity formation.
In dimension~3 this was carried out successfully by Perelman \cite{Perelman1} and led to the construction of a Ricci flow with surgery \cite{Perelman2}.
In higher dimensions, progress towards this goal was recently obtained by the author in \cite{Bamler_HK_entropy_estimates,Bamler_RF_compactness, Bamler_HK_RF_partial_regularity}, where singularities were characterized by (possibly singular) gradient shrinking solitons that arise as blow-up models along specific sequences of points and scales.
In the same work it was shown that it is --- in fact --- possible to take blow-up limits along \emph{any} sequence of points and scales and that this limit is given by a non-collapsed, ancient Ricci flow with a possible singular set of codimension $\geq 4$.
In order to obtain a characterization of the part of the manifold that becomes singular, it becomes necessary to study these non-collapsed, ancient flows in more detail. 
A better understanding of these flows may allow us in the future to perform a successful surgery construction and possibly derive useful topological consequences on the underlying manifold.

This short paper is the beginning of study of non-collapsed, ancient Ricci flows.
In dimension~3, such flows are \emph{$\kappa$-solutions,} which arose in Perelman's work \cite{Perelman2} and are now fully classified \cite{Brendle_2020_Bryant, Brendle_Daskalopoulos_Sesum_2020}.
In higher dimensions, however, very little has been known about such flows. 

In this paper we will focus, for simplicity, on non-collapsed ancient flows that are \emph{non-singular} and have \emph{complete time-slices} and \emph{bounded curvature on compact time-intervals.}
This regularity assumption was considered frequently in dimension~3, for example in the study of $\kappa$-solutions where it turned out to be natural. 
We also remark that in forthcoming work, we will establish a more general theory of \emph{non-compact, singular} flows and our techniques will readily generalize to such flows.

Let us now summarize our results.
We will show that the underlying manifold of any non-collapsed, ancient flow with complete time-slices and bounded curvature on compact time-intervals has finite fundamental group.
This generalizes what has been known to be true for $\kappa$-solutions in dimension~2, 3.
We will also show that this fundamental group must be the quotient of the fundamental group of the regular part of any tangent flow at infinity, where the latter is given by a singular gradient shrinking soliton.
This implies further restrictions on the fundamental group of the flow in certain cases.
Lastly, we show that any tangent flow at infinity is a quotient of the corresponding tangent flow at infinity of the the universal covering flow and derive an identity on the order of the fundamental group of the original flow in terms of the Nash entropies at infinity of both flows.

In forthcoming work, we will study non-collapsed, ancient flows in more detail and characterize their geometry at (spatial) infinity more precisely.
\bigskip

Recall that an ancient Ricci flow $(M, (g_t)_{t \leq 0})$ with complete time-slices and bounded curvature on compact time-intervals is called \emph{non-collapsed} if the following equivalent equivalent conditions are satisfied \cite{Chan_Ma_Zhang_2021_sobolev}:
\begin{enumerate}
\item We have the following entropy bound for some uniform $Y$
\[ \mu(g_t, \tau) \geq - Y, \qquad \text{for all} \quad t \leq 0, \; \tau > 0. \]
\item We have the entropy bound
\[ \liminf_{\tau \to \infty} \mu(g_{-\tau}, \tau) > -\infty. \]
\item For some (or any) $(x,t) \in M \times \IR_{\leq 0}$ we have the following bound on the pointed Nash entropy
\[ \liminf_{\tau \to \infty} \NN_{x,t}(\tau) > -\infty. \]
\end{enumerate}
Moreover, $(M, (g_t)_{t \leq 0})$ is automatically non-collapsed if it arises as the blow-up limit of a Ricci flow $(M', (g'_t)_{t \in [0,T)})$ on a compact manifold $M'$.
Any such non-collapsed, ancient flow has a tangent cone at infinity, which is given by the flow of a singular gradient shrinking soliton.
Such a soliton is described by a \emph{singular space,} i.e., a tuple $(X, d, \RR_X, g_X, f)$, where $(X,d)$ is a metric length space, $\RR_X \subset X$ is a maximal open subset equipped with a smooth manifold structure and Riemannian metric $g_X$ such that $(X,d)$ is the metric completion of the length structure of $(\RR_X, g_X)$ and $f$ is a smooth potential function such that the gradient shrinking soliton equation
\[ \Ric_{g_X} + \nabla^2 f - \frac12 g_X = 0 \]
holds on $\RR_X$ (see \cite[Definition~2.15]{Bamler_HK_RF_partial_regularity} for further properties).
It was proven in \cite{Bamler_HK_RF_partial_regularity} that the singular set $\SS_X = X \setminus \RR_X$ of this soliton has codimension $\geq 4$ in the Minkowski-sense; moreover, in dimension 4, $(X,d)$ is the length space of a smooth orbifold with isolated singularities.

Let us now state our results.
We first show that any non-collapsed, ancient flow must have finite fundamental group.

\begin{Theorem} \label{Thm_finite_fg}
Let $(M, (g_t)_{t \leq 0})$ be a non-collapsed, ancient Ricci flow that has complete time-slices and bounded curvature on compact time-intervals.
Then $\pi_1 (M)$ is finite.
\end{Theorem}

The next theorem relates the fundamental group of the flow with the fundamental group of the regular part $\RR_X$ of any of its tangent flows at infinity.

\begin{Theorem} \label{Thm_pi1_RR_M}
Let $(M, (g_t)_{t \leq 0})$ be a non-collapsed, ancient Ricci flow with complete time-slices and bounded curvature on compact time-intervals. 
Consider a tangent flow at infinity that is given by the singular space  $(X,d,\RR_X,g_X, f)$.
Then there is a surjective homomorphism $\phi : \pi_1 (\RR_X) \to \pi_1 (M)$.

More specifically, for any bounded, open $U \subset \RR_X$ there is a smooth embedding $\psi : U \to M$ such that the following diagram commutes:
\[\begin{tikzcd}
	{\pi_1(U)} \\
	{\pi_1(\mathcal{R}_X)} && {\pi_1(M)}
	\arrow["\phi"', from=2-1, to=2-3]
	\arrow[from=1-1, to=2-1]
	\arrow["{\psi_*}", dashed, from=1-1, to=2-3]
\end{tikzcd}\]
\end{Theorem}

A direct consequence is the following:

\begin{Corollary} \label{Cor_sc}
If in the setting of Theorem~\ref{Thm_pi1_RR_M} the regular part $\RR_X$ of a tangent flow at infinity is simply-connected, then so is $M$.
\end{Corollary}

So, for example, if the tangent flow at infinity is cylindrical, $S^{k \geq 2} \times \IR^{n-k}$, then $M$ must be simply-connected itself.

The following two examples show that the statement of Theorem~\ref{Thm_pi1_RR_M} is optimal, i.e., that we cannot expect $\phi$ to be an isomorphism.

\begin{Example}
Consider the constant flow on a Ricci flat, simply connected, 4-dimensional ALE-space $M$.
Its tangent flow at infinity is $\IR^4/ \Gamma$ for some $\Gamma \subset SO(3)$, so the fundamental group of the regular part $\RR_X$ is non-trivial, while $M$ is simply-connected.
This shows that the homomorphism $\phi$ in Theorem~\ref{Thm_pi1_RR_M} need not be injective.
\end{Example}

\begin{Example}
In \cite{Appleton:2017xy} Appleton constructs a family of non-collapsed,  4-dimensional steady solitons on the total space $M_k$ of the complex line bundle $\mathcal{O}(k)$, $k \geq 3$, over $\IC P^1$.
The end of each such soliton is diffeomorphic to $\IR_{+} \times S^3/\IZ_k$ and the metric is asymptotic to a quotient of the Bryant soliton.
It follows that the tangent flow at infinity is the quotient of the cylinder $\IR \times S^3/\IZ_k$; so it has non-trivial fundamental group, despite $M_k$ being simply-connected.
This shows that even if $\RR_X = X$ the homomorphism $\phi$ in Theorem~\ref{Thm_pi1_RR_M} need not be injective.
\end{Example}

Note, however, that we have the following:

\begin{Corollary} \label{Cor_Omega}
If in the setting of Theorem~\ref{Thm_pi1_RR_M} there is a compact domain $\Omega \subset \RR_X$ such that $\pi_1(\Omega) \to \pi_1(\RR_X)$ is an isomorphism and $\partial \Omega$ consists of simply-connected components, then the map $\phi : \pi_1(\RR_X) \to \pi_1(M)$ is an isomorphism.
\end{Corollary}

So, for example, if the tangent flow at infinity is $(S^{n-1} \times \IR) / \IZ_2$, where $\IZ_2$ acts non-trivially on both factors, then $\pi_1(M) \cong \IZ_2$.

The last theorem states that the tangent of the universal cover $\td M$ can be viewed as a normal covering over of the tangent flow of $M$; this covering corresponds to the kernel of $\phi$.

\begin{Theorem} \label{Thm_GSS_quotient}
Let $(M, (g_t)_{t \leq 0})$ be a non-collapsed, ancient Ricci flow with complete time-slices and bounded curvature on compact time-intervals. 
Consider the universal cover $\pi : \td M \to M$, equipped with the pull-back Ricci flow $\td g_t := \pi^* g_t$ and the action of $\Gamma := \pi_1(M)$ by deck transformations.
Let $(\td X,\td d,\RR_{\td X},g_{\td X}, \td f)$ be its tangent flow at infinity of $(\td M, (\td g_t)_{t \leq 0})$, taken with respect to the same sequence of blow-down scales.
Then there is an isometric action of $\Gamma$ on $(X,d)$ that preserves $\RR_{\td X}$, $g_{\td X}$, $\td f$ and such that  $(\RR_X,g_X,f)$ is isometric to $(\RR_{\td X}, g_{\td X}, \td f - \log \#\Gamma) / \Gamma$.
The image of $\pi_1(\RR_{\td X})$ under the projection map is equal to the kernel of $\phi$, so the following short exact sequence holds:
\[ 1  \xrightarrow{\quad} \pi_1(\RR_{\td X}) \xrightarrow{\quad} \pi_1(\RR_X) \xrightarrow{\;\;\phi \;\;} \pi_1(M)  \xrightarrow{\quad}  1. \]
\end{Theorem}

As a corollary, we obtain the following characterization of the order of the fundamental group.

\begin{Corollary} \label{Cor_pi1_expNN}
Let $(M, (g_t)_{t \leq 0})$ be a non-collapsed, ancient Ricci flow with complete time-slices and bounded curvature on compact time-intervals. 
Consider the universal cover $\pi : \td M \to M$, equipped with the pull-back Ricci flow $\td g_t := \pi^* g_t$.
Denote by $\NN(\infty), \td\NN(\infty)$ the pointed Nash entropies at infinity of the respective flows.
Then
\[ \# \pi_1(M) = \exp \big(\td\NN(\infty) - \NN(\infty) \big) . \]
\end{Corollary}

Note that this implies $\# \pi_1(M) \leq \exp \big({-\NN(\infty)})$.
\bigskip

\section{Proofs}
We will use the results from \cite{Bamler_HK_entropy_estimates,Bamler_RF_compactness, Bamler_HK_RF_partial_regularity}; these also hold for non-compact flows with complete time-slices and bounded curvature on compact time-intervals, as explained in Appendix~\ref{appendix}.

In the following, $(M,(g_t)_{t \leq 0})$ will denote a non-collapsed, ancient flow with complete time-slices and bounded curvature on compact time-intervals and $\pi : \td M \to M$ will denote the universal cover of $M$, equipped with the pull-back Ricci flow $\td g_t := \pi^* g_t$.
Recall that $\Gamma := \pi_1(M)$ acts on $\td M$ via deck transformations.
The conjugate heat kernels $K, \td K$ on $M, \td M$, respectively satisfy for $x = \pi (\td x), y \in M$, $\td x \in \td M$
\begin{equation} \label{eq_sum_formula_K}
 K(x,t;y,s) = \sum_{\td y \in \pi^{-1}(x)} \td K(\td x,t;\td y,s). 
\end{equation}
This sum converges due to \cite[Theorem~26.25]{Chow_book_series_Part_III}.
\bigskip

We first prove Theorem~\ref{Thm_finite_fg}.

\begin{proof}[Proof of Theorem~\ref{Thm_finite_fg}.]
Fix some basepoints $x = \pi (\td x) \in M$, $\td x \in \td M$ and write $\td x_\alpha := \alpha . \td x$ for $\alpha \in \Gamma$.
Let $s < 0$ be a constant, which we will send to $-\infty$ later.  
Choose an $H_n$-center $(\td z, s)$ of $(\td x, 0)$ (see \cite[Definition~3.10]{Bamler_HK_entropy_estimates}).
Write again $\td z_\alpha := \alpha. \td z$ for $\alpha \in \Gamma$.

\begin{Claim}
$(z,s) := (\pi(\td z),s)$ is an $H_n$-center of $(x,0)$.
\end{Claim}

\begin{proof}
This follows from the fact that $\pi$ is 1-Lipschitz and \eqref{eq_sum_formula_K}.
\end{proof}

Fix some finite subset $S \subset \Gamma$ containing the identity element.

\begin{Claim}
If $s$ is small enough, then $d_s(\td z_\alpha, \td z_\beta) \leq 3 \sqrt{H_n |s|}$ for all $\alpha, \beta \in S$.
\end{Claim}

\begin{proof}
Using \cite[Lemma~2.7]{Bamler_HK_entropy_estimates}, we obtain
\begin{multline*}
 d_s (\td z_\alpha, \td z_\beta) 
= d^{g_s}_{W_1} (\delta_{\td z_\alpha}, \delta_{\td z_\beta})
\leq d^{g_s}_{W_1} (\delta_{\td z_\alpha}, \nu_{\td x_\alpha, 0; s}) 
+ d^{g_s}_{W_1} (\nu_{\td x_\alpha, 0; s}, \nu_{\td x_\beta, 0; s})
+ d^{g_s}_{W_1} (\nu_{\td x_\beta, 0; s}, \delta_{\td z_\beta}) \\
\leq 2 \sqrt{H_n |s|} + d_0(\td x_\alpha, \td x_\beta).
\end{multline*}
The last term can be bounded by $\sqrt{H_n |s|}$ for small enough $s$, since $S$ is finite.
\end{proof}

\begin{Claim} \label{Cl_number_preimages}
If $s$ is small enough, then for any $y \in B(z,s,\sqrt{2H_n |s|})$
\[ \# \big(\pi^{-1}(y) \cap B(\td z,s,10 \sqrt{H_n|s|}) \big) \geq \# S. \]
\end{Claim}

\begin{proof}
Choose a lift $\td y \in B(\td z,s, \sqrt{2H_n|s|})$ of $y$.
For any $\alpha \in S$ and $\td y_\alpha := \alpha. \td y$ we have
\[ d_s(\td y_\alpha, \td z) \leq d_s(\td y_\alpha,  \td z_\alpha) + d_s( \td z_\alpha, \td z) 
= d_s(\td y,  \td z)+ d_s( \td z_\alpha, \td z) 
\leq \sqrt{2H_n|s|} + 3 \sqrt{H_n |s|} \leq 10 \sqrt{H_n|s|}. \qedhere \]
\end{proof}
\medskip

Now recall from \cite[Theorems~6.2, 8.1]{Bamler_HK_entropy_estimates} that for some dimensional constant $c > 0$
\begin{align*}
 \big|B(z,s,\sqrt{2H_n |s|})\big|_s &\geq c \exp (\NN_{x,0}(-s)) \\
 \big|B(\td z,s,10 \sqrt{H_n|s|})\big|_s &\leq c^{-1} \exp (\NN_{\td x,0}(-s)) 
\end{align*}
So by Claim~\ref{Cl_number_preimages}
\[ \# S \leq \frac{\big|B(\td z,s,10 \sqrt{H_n|s|}) \big|_s}{\big|B( z,s, \sqrt{2H_n|s|}) \big|_s} \leq c^{-2} \exp \big(\NN_{\td x,0}(-s) - \NN_{x,0}(-s) \big).
 \]
The result now follows by letting $s \to -\infty$ and noticing that $S \subset \Gamma$ was chosen arbitrarily.
\end{proof}
\bigskip

Next, we prove Theorems~\ref{Thm_pi1_RR_M}, \ref{Thm_GSS_quotient}.

\begin{proof}[Proof of Theorems~\ref{Thm_pi1_RR_M}, \ref{Thm_GSS_quotient}.]
Fix basepoints $\td x \in \td M$, $x = \pi(\td x) \in M$ and a sequence $\la_i \to 0$ such that the flows $(M,(g_t)_{t \leq 0})$, $(\td M, (\td g_t)_{t \leq 0})$, pointed at $(x,0)$, $(\td x, 0)$ and parabolically rescaled by $\la_i$ converge in the $\IF$-sense to singular gradient shrinking solitons that are represented by $(X,d,\RR_X,g_X, f)$ and $(\td X,\td d,\RR_{\td X},g_{\td X}, \td f)$.
This implies that we can find exhaustions $U_1 \Subset U_2 \Subset \ldots \Subset \RR_X$, $\td U_1 \Subset \td U_2 \Subset \ldots \Subset \RR_{\td X}$ and embeddings $\psi_i : U_i \to M$, $\td\psi_i : \td U_i \to \td M$ such that for $t_i = -\la_i^{-2}$
\begin{equation} \label{eq_Cinfty_conv}
\begin{alignedat}{2}
 \la_i^2 \psi_i^* g_{t_i} &\xrightarrow[i \to \infty]{C^\infty_{\loc}} g_X, \qquad \psi_i^* d\nu_{x,0;t_i} &\xrightarrow[i \to \infty]{C^\infty_{\loc}} (4\pi)^{-n/2} e^{-f} dg_X, \\
\la_i^2 \td\psi_i^* \td g_{t_i} &\xrightarrow[i \to \infty]{C^\infty_{\loc}} g_{\td X}, \qquad \psi_i^* d\nu_{\td x,0;t_i} &\xrightarrow[i \to \infty]{C^\infty_{\loc}} (4\pi)^{-n/2} e^{-\td f} dg_{\td X}.
\end{alignedat}
\end{equation}

Consider the quantity $\rrm$ on $(M,(g_t)_{t \leq 0})$, $(\td M, (\td g_t)_{t \leq 0})$ as defined in \cite[Definition~10.1]{Bamler_HK_entropy_estimates}.

\begin{Claim}\label{Cl_in_psi_image}
If for some sequence $y_i \in M$ we have
\begin{equation} \label{eq_rrm_nu_cond}
 \rrm ( y_i, t_i) \geq r \sqrt{-t_i}, \qquad
\nu_{ x, 0; t_i} \big( B( y_i, t_i, r \sqrt{-t_i}) \big) \geq c, 
\end{equation}
for some uniform $r, c > 0$, then $ y_i \in \psi_i (  U_i)$ for large $i$ and $\psi_i^{-1}(y_i)$ is precompact in $\RR_X$.
The corresponding statement is also true for the flow $(\td M, (\td g_t)_{t \leq 0})$.
\end{Claim}

\begin{proof}
By \cite[Theorem~6.1]{Bamler_HK_entropy_estimates}, \cite[Lemma~9.15]{Bamler_RF_compactness} we obtain that for any $r' \in (0,r]$ there is a $c'(r')>0$ such that $\nu_{ x, 0; t_i} ( B( y_i, t_i, r' \sqrt{-t_i}) ) > c(r')$ for large $i$.
So since $\nu_{ x,0;t_i} ( \psi_i ( U_i)) \to 1$, we obtain that $(-t_i)^{-1/2} d_{t_i} ( y_i, \psi_i ( U_i) ) \to 0$.
Fix some $y'_i \in U_i$ with 
\begin{equation} \label{eq_d_yi_psiypi}
(-t_i)^{-1/2} d_{t_i} (y_i, \psi_i(y'_i)) \to 0.
\end{equation}
Then we have $\rrm(\psi_i(y'_i), t_i) \geq \frac12 r \sqrt{-t_i}$ for large $i$, so by \cite[Lemma~15.16]{Bamler_HK_RF_partial_regularity} the ball $B( y'_i, r/4) \subset \RR_{ X}$ is relatively compact for large $i$.
Due to (\ref{eq_d_yi_psiypi}) this implies that $y_i \in \psi_i (B( y'_i, r/4))$ for large $i$.
For the last statement observe that due to the bound $\Var (\nu_{x,0;t_i}) \leq H_n (-t_i)$ and the second bound in \eqref{eq_rrm_nu_cond}, the sequence $\psi_i^{-1}(y_i)$ must be bounded.
So we can again apply \cite[Lemma~9.15]{Bamler_RF_compactness}.
\end{proof}

\begin{Claim}  \label{Cl_chi}
For any fixed $j$ we have $\pi (\td\psi_i(\td U_j)) \subset \psi_i(U_i)$ for some large $i$.
Moreover, for any $\td y' \in \RR_{\td X}$ the sequence $\psi_i^{-1}(\pi(\td\psi_i(\td y'))) \in \RR_X$ is precompact.
\end{Claim}

\begin{proof}
Suppose this was false for some fixed $j$, so after passing to a subsequence, we find $\td y'_i \in \td U_j$ such that $\pi(\td\psi_i(\td y'_i)) \not\in \psi_i(U_i)$.
By relative compactness of $\td U_j$ and the convergence \eqref{eq_Cinfty_conv} we can ensure that a condition of the form \eqref{eq_rrm_nu_cond} holds for $y_i := \pi(\td\psi_i(\td y'_i))$.
This contradicts our assumption via Claim~\ref{Cl_in_psi_image}.
The last statement follows from the last statement of Claim~\ref{Cl_in_psi_image}.
\end{proof}

So after possibly shrinking $\td U_i$, we may assume in the following that
\[ \pi (\td\psi_i(\td U_i)) \subset \psi_i(U_i). \]

\begin{Claim} \label{Cl_surj}
For any $y' \in \RR_{X}$ there is some $j$ such that we can find a sequence $\td y'_i \in \td U_j$ with $\psi_i (y') = \pi (\td\psi_i(\td y'_i))$ for large $i$.
\end{Claim}

\begin{proof}
We can choose uniform constants $r, c >0$ such that \eqref{eq_rrm_nu_cond} holds for $y_i := \psi_i (y')$ for large $i$.
We can find lifts $\td y_i \in \td M$, $y_i = \pi (\td y_i)$ such that \eqref{eq_rrm_nu_cond} holds with $c$ replaced by $c/ \# \Gamma$.
The claim now follows from Claim~\ref{Cl_in_psi_image}.
\end{proof}

\begin{Claim} \label{Cl_K_alpha_K}
For any fixed $j$ we have $\alpha . \td\psi_i(\td U_j) \subset \td\psi_i (\td U_i)$ for large $i$ and all $\alpha \in \Gamma$.
Moreover, for any $\td y_i \in \td\psi_i(\td U_i)$ and $\alpha \in \Gamma$ the sequence $\td\psi_i^{-1} ( \alpha. \td y_i)$ is precompact in $\RR_{\td X}$.
Lastly, for any $y' \in \RR_X$ and $\alpha \in \Gamma$ we have
\begin{equation} \label{eq_alph_K_K}
 \lim_{i \to \infty} \big( (-t_i)^{n/2} K(\td x, 0; \td\psi_i(y')) -(-t_i)^{n/2} K(\td x, 0; \alpha. \td\psi_i(y')) \big) = 0. 
\end{equation}
\end{Claim}

\begin{proof}
Since $\Gamma$ is finite by Theorem~\ref{Thm_finite_fg}, it suffices to consider a fixed $\alpha \in \Gamma$.
The first two statements follow by contradiction, similarly as in the proof of Claim~\ref{Cl_surj}:
Suppose that there are $\td y_i \in \td\psi_i(\td U_j)$ such that $\alpha. \td y_i \not\in \td\psi (\td U_i)$.
We can again find $r, c > 0$ such that \eqref{eq_rrm_nu_cond} holds for $\td y_i$.
Next, observe that $\rrm (\alpha . \td y_i) = \rrm (\td y_i)$.
Moreover, since for any measurable subset $S \subset \td M$ the map $q \mapsto \Phi^{-1} (\nu_{q,0;t_i}(S) )$ is $(-t_i)^{-1/2}$-Lipschitz (see \cite[Definition~3.2(6)]{Bamler_RF_compactness}), we obtain
\begin{equation} \label{eq_nu_B_similar}
 \lim_{i \to \infty} \big( \nu_{\td x, 0; t_i} \big( B( \td y_i, t_i, r \sqrt{-t_i}) \big) 
-  \nu_{\alpha^{-1} . \td x, 0; t_i} \big( B( \td y_i, t_i, r \sqrt{-t_i}) \big) \big) = 0. 
\end{equation}
So by Claim~\ref{Cl_in_psi_image} we must have $\alpha. \td y_i \in \td\psi_i (\td U_i)$ for large $i$.
Finally \eqref{eq_alph_K_K} follows from \eqref{eq_nu_B_similar} after letting $r \to 0$.
\end{proof}

Choose $\chi_i : \td U_i \to U_i$ such that $\psi_i \circ \chi_i = \pi \circ \td\psi_i$.
Due to \eqref{eq_Cinfty_conv} we have
\[ \chi_i^* g_{X} \xrightarrow[i\to\infty]{C^\infty_{\loc}} g_{\td X}, \qquad
(\chi_i)_* e^{-\td f} dg_{\td X} \xrightarrow[i\to\infty]{C^\infty_{\loc}} e^{- f} dg_{ X}. \]
By Claim~\ref{Cl_chi} the sequence $\chi_i(\td y') \in \RR_X$ is precompact for any $\td y' \in \RR_{\td X}$.
So by Arzela-Ascoli 
we obtain that, after passing to a subsequence, we have $\chi_i \to \pi'$ in $C^\infty_{\loc}$ for some $\pi' : \RR_{\td X} \to \RR_X$ satisfying
\[ (\pi')^* g_{X} = g_{\td X}, \qquad \pi'_* e^{-\td f} dg_{\td X} = e^{- f} dg_{ X}. \]
Moreover, $\pi'$ is surjective by Claim~\ref{Cl_surj}.

By the same argument, and using Claim~\ref{Cl_K_alpha_K}, we obtain that, after passing to a subsequence, the action of $\Gamma$ via deck transformations converges to a smooth, isometric action on $(\RR_{\td X}, g_{\td X})$, which is equivariant under $\pi'$, fixes $\td f$ and acts transitively on the preimages $(\pi')^{-1}(y')$, $y' \in \RR_X$.
Since $\RR_{\td X}$ is connected, it follows that $\pi'$ is a covering map.
The fact that the limiting action of $\Gamma$ on $\RR_{\td X}$ is free follows by smoothness and the fact that $\Gamma$ is finite, via a center of mass construction.

So $(\RR_X, g_X)$ is indeed the isometric quotient of $(\RR_{\td X}, g_{\td X})$ under the limiting action of $\Gamma$.
Since $\td f$ factors through $\pi'$ and since 
\[ \int_{\RR_X} (4\pi)^{-n/2} e^{-f} dg_X 
= \int_{\RR_{\td X}} (4\pi)^{-n/2} e^{-\td f} dg_{\td X} = 1, \]
we obtain that $f$ the the quotient of $\td f - \log \# \Gamma$.

We now construct the map $\phi : \pi_1 (\RR_X) \to \pi_1(M)$.
Fix basepoints $p \in \RR_{X}$, $q \in M$ and lifts $\td p \in \RR_{\td X}$, $\td q \in \td M$; we will consider all fundamental groups with respect to these basepoints.
Observe that $\pi_1(\RR_X) / \pi'_* \pi_1(\RR_{\td X})$ is naturally isomorphic to the group of deck transformations $\Gamma$ of $\pi' : \RR_{\td X} \to \RR_X$.
This produces a sequence of natural homomorphisms $\phi_i : \pi_1(\RR_X) \to \Gamma = \pi_1(M)$ with kernel $\pi'_* \pi_1(\RR_{\td X})$.
After passing to a subsequence, we may assume that $\phi_i = \phi$ is independent of $i$.
Let us describe $\phi$ in geometric terms.
Fix $i$ and consider a closed loop $\gamma : [0,1] \to U_i$ based at $p$.
Set $q'_i := \psi_i(p)$ and let $\td q'_i \in \td M$ be the lift of $q'_i$ near $\td\psi_i (\td p)$.
Then $\phi ([\gamma])$ represents the deck transformation that maps $\td q'_i$ to the endpoint of the lift of $\psi_i \circ \gamma$ starting at $\td q'_i$.
So if $\sigma_i : [0,1] \to M$ is the projection of a path between $\td q$ and $\td q'_i$, then $\phi ([\gamma])$ is represented by the concatenation of $\sigma_i * (\psi_i \circ \gamma) * \sigma^{-1}_i$.
It follows that there is an isotopy $\xi_i : M \to M$ sending $q'_i$ to $q$ such that $\phi_i = (\xi_i \circ \psi_i)_*$.
This explains the construction of the embedding $\psi$ in the commutative diagram in Theorem~\ref{Thm_pi1_RR_M} and proves the first part of Theorem~\ref{Thm_GSS_quotient}.
\end{proof}
\bigskip

Corollary~\ref{Cor_sc} is a direct consequence of Theorem~\ref{Thm_pi1_RR_M}.
Corollary~\ref{Cor_Omega} follows using Van Kampen's Theorem and the fact that the boundary components of $\psi(\Omega)$ need to be separating, because $\pi_1 (M)$ is finite.
Corollary~\ref{Cor_pi1_expNN} follows from Theorem~\ref{Thm_GSS_quotient} and the fact that the pointed Nash entropies pass to the limit.

\appendix
\section{Justification of the non-compact case} \label{appendix}
We will now explain how the techniques from \cite{Bamler_HK_entropy_estimates,Bamler_RF_compactness, Bamler_HK_RF_partial_regularity} can be generalized to the setting of Ricci flows with complete time-slices and bounded curvature on compact time-intervals.

First, note that in this series of papers global, spatial integrals are usually taken with respect to a background measure of the form $d\nu_t = (4\pi \tau)^{-n/2} e^{-f} dg_t$, possibly combined with an integration over a compact time-interval.
Here $(4\pi \tau)^{-n/2} e^{-f} = K(x_0, t_0; \cdot, \cdot)$ denotes the conjugate heat kernel based at a point $(x_0, t_0)$ and $\tau = t_0 - t$.
Due to \cite[Theorems~26.25, 26.31]{Chow_book_series_Part_III} (combined with Bishop-Gromov volume comparison) the density function of the of this measure satisfies an upper and lower Gaussian bound of the form
\begin{multline} \label{eq_HK_bound_Gau}
 \frac{C^{-1}}{|B(x_0,t_0,\sqrt{\tau})|_{t_0}} \exp \bigg( {- \frac{d^2_t(x_0,x)}{C^{-1} \tau} }\bigg) \leq (4\pi \tau)^{-n/2} e^{-f(x,t)} = K(x_0, t_0; x,t) \\ \leq \frac{C}{|B(x_0,t_0,\sqrt{\tau})|_{t_0}} \exp \bigg( {- \frac{d^2_t(x_0,x)}{C\tau} }\bigg), 
\end{multline}
where $C$ may depend on a global curvature bound on $M \times [t,t_0]$.
Observe that \eqref{eq_HK_bound_Gau} implies a quadratic growth bound on $f$.
Due to local gradient estimates (see for example \cite[Lemma~9.15]{Bamler_RF_compactness}), we also obtain Gaussian bounds of the form
\[ |\nabla^m f|^p \, (4\pi \tau)^{-n/2} e^{-f(x,t)}  \leq \frac{C}{\tau^{n/2}} \exp \bigg( {- \frac{d^2_t(x_0,x)}{C\tau} }\bigg), \]
where $C$ may now also depend on $m, p$ and on the injectivity radius at $x_0$.
Similar bounds are true if we replace $|\nabla^m f|^p$ by $|\nabla^m u|^p$ for a solution $u$ of the heat equation on a compact time-interval that grows at most polynomially in space or by terms of the form $|\nabla^m_{x_0} K(x_0, t_0; x,t)|$.
On the other hand, due to the uniform curvature bound on compact time-intervals, we have an upper volume bound on distance balls of the form $|B(p,t,r)| \leq C e^{Cr}$.
This justifies the existence of the global integrals in \cite{Bamler_HK_entropy_estimates,Bamler_RF_compactness, Bamler_HK_RF_partial_regularity} and the application of Stokes' Theorem wherever necessary.

Let us now discuss how to resolve remaining issues in the non-compact case.

The main results of \cite{Bamler_HK_entropy_estimates} were generalized to the non-compact setting in \cite{Ma_Zhang_2021_ancient,Chan_Ma_Zhang_2021_ancient}; we include a brief overview for completeness:
\begin{enumerate}[label=\arabic*.]
\item The maximum principle is applied in the proofs of \cite[Theorems~4.1, 7.2, 11.1]{Bamler_HK_entropy_estimates} and occasionally to justify a lower scalar curvature bound of the form $R \geq - \frac{n}{2t}$.
In the non-compact case these applications remain justified by \cite[Theorem~12.22]{Chow_book_series_Part_II} and local gradient estimates.
\item In the proof of \cite[Theorem~4.1]{Bamler_HK_entropy_estimates} the compactness property is used to guarantee that $u_{t_0}$ takes values in $(\eps, 1-\eps)$ for some $\eps > 0$, as well as in the application of a maximum principle.
Alternatively, we could consider the function $(1-2\eps) u + \eps$ instead of $u$ and then let $\eps \to 0$.
\item In the proof of \cite[Claim~4.6]{Bamler_HK_entropy_estimates}, we required $b_m > \max_M q$, which may be infinite in the non-compact case.
Instead, we can just use bound $\nu (\{ q > b_m\}) \leq \frac1{b_m} \int_M q \, d\nu \leq \frac{C}{b_m}$ and let $b_m \to \infty$.
\item In the proof of \cite[Theorem~7.2]{Bamler_HK_entropy_estimates}, we started an induction with the statement that given a Ricci flow defined on $[0,1]$ there are some uniform constants $Q, Z < \infty$ such that
\begin{equation} \label{eq_K_bound_ZQ}
 K(x,t;y,0) \leq \frac{2Z \exp(-\NN^*_0(x,t))}{t^{n/2}} \exp \bigg({- \frac{d^2_0(z,y)}{Qt} }\bigg) 
\end{equation}
for any $(x,t) \in (0,\frac12] \times M$ and any $H_n$-center $(z,0)$ of $(x,t)$.
To see that such a bound holds, recall that by \cite[Theorem~6.1]{Bamler_HK_entropy_estimates} and the global curvature bound we must have
\begin{equation} \label{eq_lower_vol_bound}
 |B(x,t,\sqrt{t})|_t \geq c \exp (\NN^*_0(x,t)) 
\end{equation}
for some uniform $c > 0$.
Moreover, it follows from \eqref{eq_HK_bound_Gau} that
\begin{equation} \label{eq_K_N_bound}
 K(x,t;y,0) \leq \frac{C}{|B(x,t,\sqrt{t})|_t} \exp \bigg({- \frac{d^2_0(x,y)}{Ct} }\bigg), 
\end{equation}
where $C$ is independent of $x,y,t$.
This implies a bound of the form $d_0(x,z) \leq C \sqrt{t}$, which combined with \eqref{eq_lower_vol_bound} \eqref{eq_K_N_bound} gives \eqref{eq_K_bound_ZQ}.
\item The function $h$ \cite[Theorem~11.1]{Bamler_HK_entropy_estimates} should be assumed to be in $C^1 \cap L^p(d\nu_{t_0-\tau})$ and the proof follows by approximating $h$ with bounded functions.
\item The function $u$ \cite[Theorem~12.1]{Bamler_HK_entropy_estimates} should be assumed to have at most polynomial spatial growth.
\end{enumerate}
Note that \cite[Corollary~9.6]{Bamler_HK_entropy_estimates} implies that $P^*$-parabolic neighborhoods are, in fact, relatively compact within $M \times I$, so any analysis restricted to them is still local.

For \cite{Bamler_RF_compactness} we only need to realize that non-compact Ricci flows with complete time-slices and bounded curvature on compact time-intervals satisfy the axioms of an $H_n$-concentrated metric flow.
This has already been discussed in regards to \cite{Bamler_HK_entropy_estimates}.

Let us now focus on \cite{Bamler_HK_RF_partial_regularity}:
\begin{enumerate}[label=\arabic*., start=7]
\item At various places, we use Perelman's differential Harnack inequality
\[ w = \tau(2\triangle f- |\nabla f|^2 + r) + f - n \leq 0 \]
for conjugate heat kernels $v = (4\pi\tau)^{-n/2} e^{-f}$.
This follows via the maximum principle from the inequality $\square^* (wu) = -2 \tau |\Ric + \nabla^2 - \frac1{2\tau} |^2 v \leq 0$ \cite{Perelman1} in combination with an asymptotic condition on $w, u$ as $\tau \searrow 0$ (see \cite[Section~4]{Chow_book_series_Part_II}, \cite{Ni_2006_note_Harnack}).
The application of the maximum principle remains justified by \cite[Theorem~12.22]{Chow_book_series_Part_II} and local gradient estimates.
The proof of the asymptotic condition continues to work out in the non-compact setting after some modifications of the arguments.
Alternatively, one may also express the flow restricted on a small time-interval as a limit of Ricci flows on compact manifolds, using \cite{Hamilton_RF_compactness, Cheeger_Gromov_Chopping_RM}.
Then the bound $w \leq 0$ holds for small $\tau$ due to a limit argument, so by the maximum principle it holds for all $\tau$.
\item Note that the weak splitting maps constructed in 
 \cite[Sections~10, 11]{Bamler_HK_RF_partial_regularity} arise by linear combination of potentials of conjugate heat kernels, so by \eqref{eq_HK_bound_Gau} their spatial growth is at most quadratic. 
  Similarly, the almost radial function $q$ from \cite[Proposition~13.1]{Bamler_HK_RF_partial_regularity} and the weak splitting map in the end of the proof of \cite[Proposition~13.19]{Bamler_HK_RF_partial_regularity} have at most polynomial spatial growth.
 Moreover, the strong splitting maps from \cite[Proposition~12.21]{Bamler_HK_RF_partial_regularity} are uniformly bounded by construction.
 \item In order to apply the degree argument in the proof of \cite[Proposition~16.1]{Bamler_HK_RF_partial_regularity}, we need to ensure that the maps $q_i$ are proper.
 This is the case due to \eqref{eq_HK_bound_Gau} if the strong splitting maps $y^i_j$ are bounded, which we may always assume.
 \item The terms $\square |\omega_l|$ in \cite[Section~17]{Bamler_HK_RF_partial_regularity} may not satisfy good bounds at spatial infinity; however, they are bounded from above by $|\square \omega_l|$, which does.
This implies that the integral of $\square |\omega_l|$ against some conjugate heat kernel measure is defined, but may be infinite.
This, or a comparable term, occurs in an integration-by-parts argument only in \cite[(17.23), (17.24), Lemmas~17.37, 17.38]{Bamler_HK_RF_partial_regularity}.
Each time, we may instead first compute integrals of the form $\int \int_M \square (|\omega_L| \eta_L) d\nu_t dt$ for some spatial cutoff function $\eta_L$ with compact support in some ball of radius $\approx L$ and $|\nabla \eta_L | \leq C/L$, $|\nabla^2 \eta_L| \leq C e^{CL}$ and then let $L \to \infty$.
\end{enumerate}

%

\bibliography{bibliography}{}
\bibliographystyle{amsalpha}

\end{document}